\documentclass[12pt]{article}
\usepackage{amssymb,amsmath,amsthm,bbm,multicol}
\usepackage{graphicx, color}

\usepackage{geometry}
\usepackage{setspace,tikz}
\usetikzlibrary{patterns,calc,decorations.markings}

\usepackage{ragged2e}

\usepackage{enumerate}

\usepackage{pgfplots}
\usepgfplotslibrary{polar}

\usepackage{xcolor}

\definecolor{olive}{rgb}{0.3, 0.4, .1}
\definecolor{fore}{RGB}{249,242,215}
\definecolor{back}{RGB}{51,51,51}
\definecolor{title}{RGB}{255,0,90}
\definecolor{dgreen}{rgb}{0.,0.4,0.}
\definecolor{gold}{rgb}{1.,0.84,0.}
\definecolor{JungleGreen}{cmyk}{0.99,0,0.52,0}
\definecolor{bgreen}{cmyk}{0.85,0,0.33,0}
\definecolor{RawSienna}{cmyk}{0,0.72,1,0.45}
\definecolor{Magenta}{cmyk}{0,1,0,0}

\newtheorem{theorem}{Theorem}

\newtheorem{lemma}[theorem]{Lemma}

\newtheorem{claim}[theorem]{Claim}

\newtheorem{corollary}[theorem]{Corollary}

\theoremstyle{remark}

\newcommand{\E}{E^*}
\newcommand{\V}{V^*}
\newcommand{\T}{T^*}
\newcommand{\w}{p}
\renewcommand{\H}{g}
\newcommand{\X}{X^*}
\newcommand{\f}{i}
\newcommand{\R}{\mathbb{R}}

\newcommand{\F}{\mathbb{F}}

\title{An elementary exposition to \\
topological overlap in the plane}

\author{Amir Yehudayoff\footnote{Department
of mathematics, Technion-IIT, Israel.
\texttt{amir.yehudayoff@gmail.com}.}}

\date{}

\begin{document}

\maketitle

\begin{abstract}
The aim of this text is to provide an elementary and self-contained
exposition of Gromov's argument
on topological overlap
(the presentation is based on Gromov's work,
as well as two follow-up papers of 
Matou\v{s}ek and Wagner,
and of Dotterrer, Kaufman and Wagner).
We also discuss a simple generalization
in which the vertices are weighted according to some probability distribution.
This allows to use von Neumann's minimax theorem
to deduce a dual statement.
\end{abstract}

\section{Introduction}

Consider $n$ generic points in $\R^2$,
and the ${n \choose 3}$ triangles they define.
Boros and Furedi~\cite{zbMATH03958121} proved that there
always exists a point $r \in \R^2$ that belongs to a fraction of at least
$2/9$ of the triangles (and that $2/9$ is tight).

Gromov~\cite{zbMATH05800304} introduced
a topological framework, that generalizes
Boros and Furedi's affine framework.
Let $X = (V,E,T)$ be the $2$-skeleton of the $n$-simplex. 
That is, the vertex set $V$ is a set of size $n$,
the edge set $E$ is the set of all subsets of $V$ of size two,
and the triangle set $T$ is the set of all subsets of $V$ of size three.
Let $f:X \to \R^2$ be a continuous map. Namely,
the image of a vertex $f(v)$ is a point in $\R^2$,
the image of an edge
$f(e) = f(\{v_1,v_2\})$ is a continuous path between $f(v_1)$ and $f(v_2)$
that is topologically an interval, and the image of a triangle $f(t) = f(\{v_1,v_2,v_3\})$ is topologically a triangle.

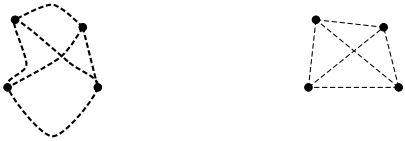
\begin{center}
\begin{tikzpicture}

\draw[fill] (0.6,-0.2) circle (0.05cm);
\draw[fill] (-0.5,0.7) circle (0.05cm);
\draw[fill] (0.4,0.6) circle (0.05cm);
\draw[fill] (-0.6,-0.2) circle (0.05cm);
\draw[thick, dash pattern=on 2pt off 1pt]  plot [smooth cycle] coordinates {(-0.5,0.7) (0.22,0.13) (0.6,-0.2) (0,-0.85) (-0.6,-0.2) (-0.35,0.1)};
\draw[thick, dash pattern=on 2pt off 1pt]  plot [smooth ] coordinates {(0.4,0.6) (0,0.9) (-0.5,0.7)};
\draw[thick, dash pattern=on 2pt off 1pt]  plot [smooth ] coordinates {(0.4,0.6) (0.5,0.2) (0.6,-0.2)};
\draw[thick, dash pattern=on 2pt off 1pt]  plot [smooth ] coordinates {(0.4,0.6) (0.1,0.2) (-0.6,-0.2)};

\draw[fill] (4.6,-0.2) circle (0.05cm);
\draw[fill] (3.5,0.7) circle (0.05cm);
\draw[fill] (4.4,0.6) circle (0.05cm);
\draw[fill] (3.4,-0.2) circle (0.05cm);
\draw[dash pattern=on 2pt off 1pt] (4.6,-0.2) -- (3.5,0.7);
\draw[dash pattern=on 2pt off 1pt] (4.6,-0.2) -- (4.4,0.6);
\draw[dash pattern=on 2pt off 1pt] (4.6,-0.2) -- (3.4,-0.2);
\draw[dash pattern=on 2pt off 1pt] (3.5,0.7) -- (4.4,0.6);
\draw[dash pattern=on 2pt off 1pt] (3.5,0.7) -- (3.4,-0.2);
\draw[dash pattern=on 2pt off 1pt] (3.4,-0.2) -- (4.4,0.6);

\end{tikzpicture}

\

\begin{minipage}{16cm}
{\em 
\noindent
Two copies of the $2$-skeleton of the $4$-complex
(four points, six edges, and four triangles).
The left one demonstrates the topological framework,
and the right one the affine framework.}
\end{minipage} 

\end{center}

Gromov proved the following
generalization of Boros and Furedi's result.

\begin{theorem}[\cite{zbMATH05800304}]
\label{thm:gromov}
For every continuous $f:X \to \R^2$,
there is a point $r \in \R^2$ 
so that the number of triangles $t \in T$ for which
$r \in f(t)$ is at least $c {n \choose 3}$ with $c = \frac{2}{9} - \frac{3}{n}$.
\end{theorem}

Consider the following generalization in which
there is a probability distribution $p$ on the vertex set $V$,
instead of the uniform distribution on $V$ in Theorem~\ref{thm:gromov}.
Linearly extend $p$ to a probability distribution on $E$ and $T$ by
$$\w(\{v_1,v_2\}) = \frac{\w(v_1)+\w(v_2)}{{n-1 \choose 1}}
\ \ \text{and} \ \
\w(\{v_1,v_2,v_3\}) = \frac{\w(v_1)+\w(v_2)+\w(v_3)}{{n-1 \choose 2}}.$$

To make the text self-contained, we focus on the case that $X$ is well-behaved in the sense defined and explained in 
Section~\ref{sec:dualityAndInter} below.
This assumption is not stringent
(for methods that allow to remove this assumption,
see~\cite{zbMATH05800304,Dotterrer:2015aa,zbMATH03287444}
and references within).

\begin{theorem}
\label{thm:main}
For every continuous well-behaved $f :X \to \R^2$
and for every distribution $p$ on $V$,
there is a point $r \in \R^2$ so that 
$\sum_{t \in T: r \in f(t)} \w(t) \geq \frac{1}{13}-\frac{3}{13(n-1)}$.
\end{theorem}

One reason to consider arbitrary distributions is
that von Neumann's minimax theorem~\cite{Neumann1928}
implies the following dual statement.

\begin{corollary}
For every continuous well-behaved $f:X \to \R^2$,
there is a distribution $\mu$ on $\R^2$ 
so that for every $v \in V$, we have
$\sum_{t \in T: v \in t}
\mu ( f(t) ) \geq \frac{1}{13}-\frac{3}{13(n-1)}$.
\end{corollary}

The proof of Theorem~\ref{thm:main} is based on Gromov's argument~\cite{zbMATH05800304},
as well as two follow up papers of 
Matou\v{s}ek and Wagner~\cite{zbMATH06346215},
and of Dotterrer, Kaufman and Wagner~\cite{Dotterrer:2015aa},
which made Gromov's work more accessible.

The aim is to keep this text as elementary and self-contained
as possible, while explaining the key ideas.
It thus avoids some of the algebraic topology 
that appears in these works,
and focuses on the $2$-dimensional case.
This is also the reason for our assumption that $f$ is well-behaved.

\paragraph{Outline of proof.}
The proof is based on a notion we call a folding map (in \cite{Dotterrer:2015aa}, the analogous notion in the proof
is called ``chain-cochain homotopy'' which is 
defined in algebraic terms using the language of commutative diagrams).
Folding maps are defined in Section~\ref{sec:proof}.
Their definition is based on a dense and generic 
triangulation $\X$ of the plane,
and on the duality/intersection structure between $\X$ and 
the image of $f$.
As it turns out, it is straight forward to verify that folding maps do not exist.
Gromov's argument is thus about showing that
if the theorems above are false then we can actually
construct a folding map. This beautiful construction is described in Section~\ref{sec:consruction}, and is based
on simple combinatorial-topological properties of $X$.
Folding maps can thus be thought of as an algebraic-topological 
obstruction to the existence of a continuous
map $f$ that violates Theorems~\ref{thm:gromov}
and~\ref{thm:main}.


\paragraph{The constants.}
The constant $2/9$ in Theorem~\ref{thm:gromov} comes in Gromov's proof from the equation
$2/9 = (1/3) \cdot (1-(1/3))$ where the $1/3$
comes from the $3$ in Lemma~\ref{lem:1/3} below
and the function $x (1-x)$ comes from Lemma~\ref{lem:x(1-x)}
below. Both lemmas describe simple combinatorial-topological 
properties of $X$ as in~\cite{zbMATH05565381,zbMATH05800304}.

The constant $1/13$ in Theorem~\ref{thm:main} is not as good as the $2/9$ in Theorem~\ref{thm:gromov}, 
and is probably not optimal
(the proof we present can be simply changed to yield the $2/9$ in Theorem~\ref{thm:gromov}).
Roughly speaking, there are two reasons for this quantitative loss.
One is that in the uniform distribution on $V$ all edges have very small weight, whereas for general distributions this may not be the case.
Another is the structure of the uniform distribution
underlying Lemmas~\ref{lem:x(1-x)} and~\ref{lem:1/3},
which is weaker for general distribution (see Lemmas~\ref{lem:MW+Ggraphsw} and~\ref{lem:MW+Gw}).

\section{Finding a point that is covered many times}
\label{sec:proof}

\paragraph{Preliminaries.}
For a finite set $Y$, we sometimes think of $\{0,1\}^Y$ as a vector space
over the field with two elements $\F_2$, and vice versa. 
Addition of subsets of $Y$ is defined accordingly.
If $h$ is a map from a set $W$ to $\F_2^Y$,
then, usually, we linearly extend $h$ to a map from $\F_2^W$ to $\F_2^Y$.

\subsection{Well-behaved Poincar\'e duality}
\label{sec:dualityAndInter}

Denote by $B$ the euclidean unit ball in the plane.
Assume without loss of generality that $f(e)$ is contained in $B/2$
for all $e \in E$.
Recall that $\X =(\V,\E,\T)$ is a triangulation of $B$ if
the vertices $\V$ is a set of points in $B$, 
the edges $\E$ is a set of line segments connecting points in $\V$, 
and the triangles $\T$ are defined by $\E$
so that they (almost) form a partition of $B$.

We say that $f$ is {\em well-behaved} if there is a finite triangulation
$\X$ of $B$ satisfying the properties listed below
(indeed if $f(E)$ consists of smooth enough
curves then it is simple to construct such $\X$).

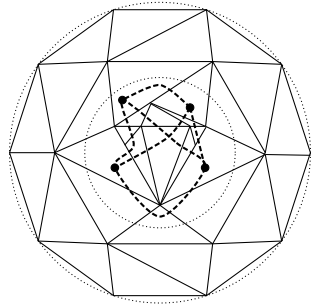
\begin{center}
\begin{tikzpicture}

\draw[fill] (0.6,-0.2) circle (0.05cm);
\draw[fill] (-0.5,0.7) circle (0.05cm);
\draw[fill] (0.4,0.6) circle (0.05cm);
\draw[fill] (-0.6,-0.2) circle (0.05cm);
\draw[thick, dash pattern=on 2pt off 1pt]  plot [smooth cycle] coordinates {(-0.5,0.7) (0.22,0.13) (0.6,-0.2) (0,-0.85) (-0.6,-0.2) (-0.35,0.1)};
\draw[thick, dash pattern=on 2pt off 1pt]  plot [smooth ] coordinates {(0.4,0.6) (0,0.9) (-0.5,0.7)};
\draw[thick, dash pattern=on 2pt off 1pt]  plot [smooth ] coordinates {(0.4,0.6) (0.5,0.2) (0.6,-0.2)};
\draw[thick, dash pattern=on 2pt off 1pt]  plot [smooth ] coordinates {(0.4,0.6) (0.1,0.2) (-0.6,-0.2)};

\draw (0:2) \foreach \x in {36,72,...,359} {
                -- (\x:2) 
                } -- cycle;
\draw (0:1.4) \foreach \x in {60,120,...,359} {
                -- (\x:1.4) 
                } -- cycle;

 \draw (-0.25,0.36)--(-0.12,0.65);
 \draw (0.4,0.355)--(-0.12,0.65);
                   
\draw (-0.25,0.36)--(0,-0.7);
\draw (0.4,0.355)--(0,-0.7);

\draw (-0.25,0.36)--(-0.47,0.1);
\draw (0.4,0.355)--(0.47,0.1);

 \draw (30:0.7) -- (101.5:0.68);               
\draw (30:0.7) -- (150:0.7);
\draw (150:0.7) -- (270:0.7);
\draw (270:0.7) -- (29:0.7);
\draw (30:0.7) -- (359:1.4);
\draw (30:0.7) -- (60:1.4);
\draw (150:0.7) -- (60:1.4);
\draw (150:0.7) -- (120:1.4);
\draw (150:0.7) -- (180:1.4);
\draw (270:0.7) -- (180:1.4);
\draw (270:0.7) -- (240:1.4);
\draw (270:0.7) -- (300:1.4);
\draw (270:0.7) -- (359:1.4);

\draw (0:1.4) -- (36:2) ;
\draw (60:1.4) -- (36:2) ;
\draw (60:1.4) -- (72:2) ;
\draw (120:1.4) -- (72:2) ;
\draw (120:1.4) -- (108:2) ;
\draw (120:1.4) -- (144:2) ;
\draw (180:1.4) -- (144:2) ;
\draw (180:1.4) -- (180:2) ;
\draw (180:1.4) -- (216:2) ;
\draw (240:1.4) -- (216:2) ;
\draw (240:1.4) -- (252:2) ;
\draw (240:1.4) -- (288:2) ;
\draw (300:1.4) -- (288:2) ;
\draw (300:1.4) -- (324:2) ;
\draw (359:1.4) -- (324:2) ;
\draw (359:1.4) -- (359:2) ;
 \draw[dash pattern=on \pgflinewidth off 1pt] (0,0) circle (2cm);
 \draw[dash pattern=on \pgflinewidth off 1pt] (0,0) circle (1cm);
\end{tikzpicture}

\

\begin{minipage}{16cm}
{\em 
\noindent
The large dotted circle is the boundary of $B$.
The small dotted circle is the boundary of $B/2$.
The dots and dashed curves are $f(V)$ and $f(E)$ with $n = 4$.
There are four topological triangles.
The image of $f$ is contained in $B/2$.
The triangulation $\X$ of $B$ shows that $f$ is well-behaved.
}
\end{minipage} 

\end{center}

\

\

This assumption is used in Lemmas~\ref{lem:interEdges}
and~\ref{lem:interTriangles} below,
and also in the proof of Theorem~\ref{thm:main}.
Intuitively, $\V$ is defined to be a set of dense and generic points. 
Specifically, the edges in $\E$ are short enough
and the triangles in $\T$ are small enough
so that ``each edge $e^*$ or triangle $t^*$ sees at most one complex
region of $\X$.''

The following properties are assumed to hold (there is some duality
between them):

\begin{description}

\item[$\cdot$] $f(v) \not \in e^*$
and $v^* \not \in f(e)$
for all $v,e,v^*,e^*$.

\item[$\cdot$] For every $v \in V$, 
there is a triangle $t^*_v \in \T$ so that $f(v) \in t^*_v$.
For every $u \neq v$ in $V$,
the point $f(u)$ is not in $t^*_v$.

\item[$\cdot$] 
For every $e \in E$ and $t^* \in \T$, 
if $f(e)$ has one end inside $t^*$
and one end outside $t^*$, then $f(e)$
intersects exactly one edge of $t^*$. \

\vspace{0.6cm}

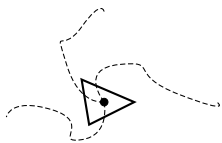
\begin{multicols}{2}
\begin{center}
\begin{tikzpicture}
\draw[fill] (0,0) circle (0.05cm);
\draw[dash pattern=on 2pt off 1pt] plot [smooth, tension=2] coordinates { (0,0) (0.2,0.5) (1,0.1) (1.5,0)};
\draw[dash pattern=on 2pt off 1pt] plot [smooth, tension=2] coordinates { (0,0) (-0.3,-0.5) (-0.7,-0.1) (-1.3,-0.2)};
\draw[dash pattern=on 2pt off 1pt] plot [smooth, tension=2] coordinates { (0,0) (-0.5,0.5) (-0.3,1) (0,1.2)};
\draw[thick] (-0.3,0.3) -- (0.4,0) -- (-0.2,-0.3) -- cycle ;
\end{tikzpicture}
\end{center}
{\em The dot is $f(v)$ for some $v \in V$,
and the triangle is $t^*_v$.
The dashed curves are images of edges containing $v$;
these edges intersect a single edge of $t^*_v$
and end outside $t^*_v$.}
\end{multicols}

\medskip

\item[$\cdot$] 
For every $e^* \in \E$ and $t \in T$, 
if $e^*$ has one end inside $f(t)$
and one end outside $f(t)$, then $e^*$
intersects exactly one edge of $f(t)$. 

\vspace{0.6cm}

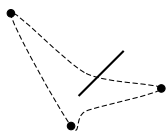
\begin{multicols}{2}
\begin{center}
\begin{tikzpicture}
\draw[fill] (-1,1) circle (0.05cm);
\draw[fill] (1,0) circle (0.05cm);
\draw[fill] (-0.2,-0.5) circle (0.05cm);
\draw[thick] (-0.1,-0.1) -- (0.5,0.5);
\draw[dash pattern=on 2pt off 1pt]  plot [smooth cycle] coordinates {(-1,1) (0,0.2) (1,0) (0,-0.3) (-0.2,-0.5)};
\end{tikzpicture}
\end{center}

{\em The dashed triangle is $f(t)$,
and the line segment is $e^*$.}
\end{multicols}

\medskip

\item[$\cdot$] 
For $v \in V$ and $e \in E$,
if $f(e) \cap t^*_v \neq \emptyset$ then $v \in e$.

\vspace{0.6cm}

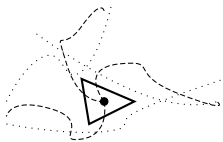
\begin{multicols}{2}
\begin{center}
\begin{tikzpicture}
\draw[fill] (0,0) circle (0.05cm);
\draw[dash pattern=on 2pt off 1pt] plot [smooth, tension=2] coordinates { (0,0) (0.2,0.5) (1,0.1) (1.5,0)};
\draw[dash pattern=on 2pt off 1pt] plot [smooth, tension=2] coordinates { (0,0) (-0.3,-0.5) (-0.7,-0.1) (-1.3,-0.2)};
\draw[dash pattern=on 2pt off 1pt] plot [smooth, tension=2] coordinates { (0,0) (-0.5,0.5) (-0.3,1) (0,1.2)};
\draw[dotted] plot [smooth, tension=1] coordinates {  (-0.9,0.9) (0.5,0.1) (1.5,0)};
\draw[dotted] plot [smooth, tension=1] coordinates {  (-1.3,-0.3)  
(0,-0.4) (0.5,-0.1) (1.6,0.3)};
\draw[dotted] plot [smooth, tension=1] coordinates {  (-1.3,0)  (-0.8,0.6)
(-0.3,0.4) (0.1,1.3)};
\draw[thick] (-0.3,0.3) -- (0.4,0) -- (-0.2,-0.3) -- cycle ;
\end{tikzpicture}
\end{center}
{\em The dot is $f(v)$,
and the triangle is $t^*_v$.
The dashed curves are images of edges containing $v$.
The dotted curves are images of edges not containing $v$;
these do not touch $t^*_v$.}
\end{multicols}

\item[$\cdot$] 
For every $e^* \in \E$ and $t \in T$, 
if $e^*$ does not have one end inside $f(t)$
and one end outside $f(t)$, then 
$e^*$ intersects at most two edges of $f(t)$. 
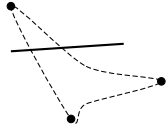
\begin{multicols}{2}
\begin{center}
\begin{tikzpicture}
\draw[fill] (-1,1) circle (0.05cm);
\draw[fill] (1,0) circle (0.05cm);
\draw[fill] (-0.2,-0.5) circle (0.05cm);
\draw[thick] (-1,0.4) -- (0.5,0.5);
\draw[dash pattern=on 2pt off 1pt]  plot [smooth cycle] coordinates {(-1,1) (0,0.2) (1,0) (0,-0.3) (-0.2,-0.5)};
\end{tikzpicture}
\end{center}
{\em The dashed topological triangle is $f(t)$,
and the line segment is $e^*$.}
\end{multicols}

\item[$\cdot$]
For every $v \in V$ and $e^* \in \E$ so that
$e^*$ is not contained in $t^*_v$,
there is at most one $u \in V$
so that $e^* \cap f(\{v,u\}) \neq \emptyset$.
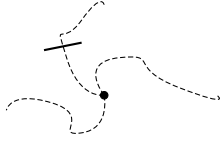
\begin{multicols}{2}
\begin{center}
\begin{tikzpicture}
\draw[fill] (0,0) circle (0.05cm);
\draw[dash pattern=on 2pt off 1pt] plot [smooth, tension=2] coordinates { (0,0) (0.2,0.5) (1,0.1) (1.5,0)};
\draw[dash pattern=on 2pt off 1pt] plot [smooth, tension=2] coordinates { (0,0) (-0.3,-0.5) (-0.7,-0.1) (-1.3,-0.2)};
\draw[dash pattern=on 2pt off 1pt] plot [smooth, tension=2] coordinates { (0,0) (-0.5,0.5) (-0.3,1) (0,1.2)};
\draw[thick] (-0.8,0.6) -- (-0.3,0.7) ;
\end{tikzpicture}
\end{center}
{\em The dot is $f(v)$, and the dashed curves are images
of edges containing $v$.
The line segment is $e^*$ that is not contained
in $t^*_v$; it touches the image of at most
one edge containing $v$.}
\end{multicols}

\item[$\cdot$] For every $e \in E$ and $t^* \in \T$,
the path $f(e)$ intersects at most two edges of $t^*$.
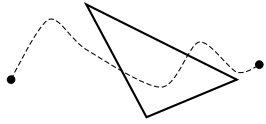
\begin{multicols}{2}
\begin{center}
\begin{tikzpicture}
\draw[fill] (-2,0) circle (0.05cm);
\draw[fill] (1.3,0.2) circle (0.05cm);
\draw[thick] (-1,1) -- (1,0) -- (-0.2,-0.5) -- cycle;
\draw[dash pattern=on 2pt off 1pt]  plot [smooth] coordinates {(-2,0) (-1.5,0.8) (-1,0.4) (0,-0.1) (0.5, 0.5) (1,0.1) (1.3,0.2)};
\end{tikzpicture}
\end{center}
{\em The dashed curve is $f(e)$,
and the triangle is $t^*$.}
\end{multicols}


\end{description}

\subsection{The intersection map}

Define $\f$ as the following intersection map
(this map is denoted\footnote{The symbol $\pitchfork$ seems to be a drawing of an intersection.} by $f^\pitchfork$ in \cite{zbMATH05800304,zbMATH06346215,Dotterrer:2015aa}):
\begin{description}
\item[$\cdot$] $\f(v^*)$ is the set of $t \in T$ so that $v^* \in f(t)$.
\item[$\cdot$] $\f(e^*)$ is the set of $e \in E$ so that $e^* \cap f(e) \neq \emptyset$. 
\item[$\cdot$] $\f(t^*)$ is the set of $v \in V$ so that $f(v) \in t^*$.
\end{description}

\subsection{Folding maps}

A map $\H : \E \to \F_2^V$ is called a {\em folding} if
the following two properties hold:
\begin{enumerate}
\item $\H(e^*_1)+\H(e^*_2)+\H(e^*_3) = \f(t^*)$ 
for every triangle $t^* \in T^*$
where $e^*_1,e^*_2,e^*_3$ are the three edges of $t^*$.
\item $\H(e^*) = 0$ for every edge $e^* \in \E$ so that
$e^* \cap (B/2) = \emptyset$.
\end{enumerate}
Gromov~\cite{zbMATH05800304} described a procedure
for constructing a folding. This procedure
yields the following (Theorem~\ref{thm:folding} is proved in Section~\ref{sec:consruction}).

\begin{theorem}
\label{thm:folding}
Assume that $\w(\f(v^*)) \leq c$ for all $v^* \in \V$
and that $\w(v) \leq c$ for all $v \in V$,
where $c = \frac{1}{13}-\frac{3}{13(n-1)}$.
Then, there is a folding map $\H: \E \to \F_2^V$.
\end{theorem}

\begin{proof}[Proof of Theorem~\ref{thm:main} given
Theorem~\ref{thm:folding}]
First, if there exists $v \in V$ so that $\w(v) \geq c$, 
then every triangle that contains $v$ has weight at least $\w(v)/{n-1 \choose 2}$
which means that $\sum_{t \in T : v \in T} \w(t)
\geq c$ and we may choose $r = f(v)$.
Second, assume towards a contradiction that Theorem~\ref{thm:main} is false.
On one hand, the sum\footnote{This sum replaces the notion ``fundamental homology class'' in~\cite{zbMATH05800304,zbMATH06346215,Dotterrer:2015aa}.}
 $\sum_{t^* \in T^*} \f(t^*) \in \F_2^V$
is the all 1 function.
On the other hand, by the lemma, 
each term $\f(t^*)$ may be replaced by a sum over edges. 
Summing over $e^*$ rather than on $t^*$,
every edge inside $B/2$ is counted twice, so contributes zero,
and every edge outside $B/2$ contributes zero as well.
Overall $\sum_{t^* \in T^*} \f(t^*) =0$, a contradiction.
\end{proof}

\section{Constructing a folding map}
\label{sec:consruction}

\subsection{Coboundary map}
\label{sec:coMap}

Define the {\em coboundary} map\footnote{The notation does not indicate if $\delta$ acts
on sets of vertices or edges, but this is clear from the context.}
$\delta$ as follows.
For $u \in  V$, 
define $\delta u$ as the set of edges $e \in E$ 
so that $u \in e$.
For $e \in E$, define
$\delta e$ as the set of triangles $t \in T$ so that
$e \in t$.
Recall that we linearly extend $\delta$ to 
$\delta: \F_2^V \to \F_2^E$ and $\delta: \F_2^E \to \F_2^T$.
It follows that if $U \subseteq V$
then $\delta U$ is the set of edges $e \in E$ 
so that $|e \cap U|$ is odd,
and if $F \subseteq E$, then
$\delta F$ is the set of triangles $t \in T$ so that
$|t \cap F|$ is odd.

\paragraph{Kernel.}
The following claims describe some simple, well-known and useful properties of the kernel of the coboundary map
(in~\cite{Kaufman:2014aa,Dotterrer:2015aa} these claims are replaced by more general ``cosystolic'' inequalities).

\begin{claim}
\label{clm:sysVertices}
Let $U \subseteq V$ be non empty.
If $\delta U = 0$ then $U = V$.
\end{claim}

\begin{proof}
If $v \not \in U$ then
all edge of the form $\{v,u\}$
for $u \in U$ are in $\delta U$.
\end{proof}

\begin{claim}
\label{clm:sysEdges}
Let $F \subseteq E$.
If $\delta F = 0$ then $F = \delta U$
for some $U \subseteq V$.
\end{claim}

\begin{proof}
Assume $F \neq 0$.
Consider the graph defined by $F$.

First, if there is an isolated vertex in the graph, then together
with one of the edges in $F$ we get a triangle in $\delta F$.
So there are no isolated vertices.

Second, we prove that the graph defined by $F$ is bipartite.
Indeed, assume towards a contradiction that it contains a cycle of odd length.
Let
$$\{v_0,v_1\},\{v_1,v_2\},\ldots,\{v_k,v_0\}$$
be an odd cycle of minimum length.
Minimality implies that there are no inner edges in the cycle.
Since $\delta F = 0$,
we know that $k \geq 4$.
So the triangle $\{v_0,v_2,v_3\}$ contains one edge from
$F$ and is in $\delta F$, a contradiction.

Third,
let $U_1,U_2 \subset V$ be the two color classes of 
the graph defined by $F$.
If $u_1 \in U_1$ is not connected to some $u_2 \in U_2$,
then $u_2$ together with one of the edges containing $u_1$
form a triangle in $\delta F$. So $F = \delta U_1$.
\end{proof}

\paragraph{Expansion.}
The following two lemmas
(and generalizations of them)
were proved by
Meshulam and Wallach~\cite{zbMATH05565381}
and Gromov~\cite{zbMATH05800304}.
The lemmas describe expansion properties
of the coboundary map
(the first lemma is extremely simple, while the second
is not).

\begin{lemma}[\cite{zbMATH05800304}]
\label{lem:x(1-x)}
Let $U \subseteq V$. 
There is $U_0 \subseteq V$ so that
$\delta U_0 = \delta U$, and
$|U_0| (n-|U_0|) = |\delta U|$
and $|U_0| \leq n/2$.
\end{lemma}

\begin{lemma}[\cite{zbMATH05565381,zbMATH05800304}]
\label{lem:1/3}
Let $F \subseteq E$. 
There is $F_0 \subseteq E$ so that
$\delta F_0 = \delta F$, and
$|F_0| \leq 3|\delta F|$.
\end{lemma}

We do not use the two lemmas above.
Instead, we use the following two lemmas 
(that replace the uniform distribution with a general distribution).
The proofs of the two lemmas below
are similar to the proofs of the two lemmas above.

\begin{lemma}
\label{lem:MW+Ggraphsw}
Let $U \subseteq V$. 
There is $U_0 \subseteq V$ so that
$\delta U_0 = \delta U$, and
$\w(U_0) < \w(\delta U)$.
\end{lemma}

\begin{proof}
We may assume $U \not \in \{\emptyset,V\}$.
The set $\delta U$ is the set of all edges between $U$ and $V \setminus U$.
Thus,
$$\w(\delta U) = \sum_{u \in U} \sum_{v \not \in U} \frac{\w(u)+\w(v)}{n-1}
= \w(U) \frac{n-|U|}{n-1} + \w(V \setminus U) \frac{|U|}{n-1}.$$
Choose $U_0 \in \{U,V\setminus U\}$ 
so that $p(U_0) \leq p(V \setminus U_0)$.
\end{proof}

\begin{lemma}
\label{lem:MW+Gw}
Let $F \subseteq E$. 
There is $F_0 \subseteq E$ so that
$\delta F_0 = \delta F$, and
$\w(F_0) < \frac{3}{2} \w(\delta F)$.
\end{lemma}

\begin{proof}
Write
\begin{align*}
3 \w(\delta F)
& = \sum_{v \in V} \sum_{t \in \delta F : v \in t} \w(t) .
\end{align*}

For a fixed $v \in V$, 
denote by $G_v$ the set of edges $e = \{v_1,v_2\}$ so that $v \not \in e$,
and $\{v,v_1,v_2\} \in \delta F$.
Every $t$ so that $v \in t \in \delta F$
corresponds to an edge $e \in G_v$ and vice versa.
Denote by $N_v$ the $F$-neighborhood of $v$, that is,
the set of vertices $u \in V$ so that $\{v,u\} \in F$.

We claim that $G_v = F+\delta N_v$.
Indeed, if $e =\{v_1,v_2\} \in G_v$,
then: If $e \in F$ then either $v_1,v_2 \in N_v$ or $v_1,v_2 \not \in N_v$,
so $e \not \in \delta N_v$.
And if $e \not \in F$ then $|\{v_1,v_2\} \cap N_v|=1$
and $e \in \delta N_v$.
On the other hand, if $e =\{v_1,v_2\} \in F+\delta N_v$, then:
If $e \in F$ then $e \not \in \delta N_v$, so $v \not \in e$
and $|\{v_1,v_2\}\cap N_v|$
is even. Thus, $\{v,v_1,v_2\} \in \delta F$.
And if $e \in \delta N_v$ then
$e \not \in F$, $v \not \in e$,
and only one of $v_1,v_2$ is connected to $v$ in the
graph $F$ defines. 
Again $\{v,v_1,v_2\} \in \delta F$.

Thus,
\begin{align*}
3 \w( \delta F)
 = \sum_{v \in V} \sum_{e \in F + \delta N_v}
 \frac{\w(v)}{{n-1 \choose 2}} + \frac{2 \w(e)}{n-2}
\geq \frac{2}{n-2} \sum_{v \in V} \w(F+\delta N_v) .
\end{align*}
Hence, there is $v$ so that
$$\w(F+\delta N_v) \leq \frac{3 (n-2)\w(\delta F)}{2n}.$$
Set $F_0 = F+ \delta N_v$.
The proof is complete since $\delta (\delta N_v) = 0$.
\end{proof}

\subsection{Intersection map}

The following lemmas 
describe the interaction between the coboundary
map and the intersection map
(Poincar\'e duality).
They follow from the assumptions on $\X$
made in Section~\ref{sec:dualityAndInter}
(a more general treatment is given in~\cite{zbMATH03287444}, 
see also \cite{Dotterrer:2015aa}).

\begin{lemma}
\label{lem:interEdges}
Let $e^* = \{v^*,u^*\} \in \E$.
Then,
$$\f(v^*)+\f(u^*) = \delta \f(e^*).$$
\end{lemma}

\begin{proof}
There are two containments to prove.
First, if $t \in \f(v^*)+\f(u^*)$ 
then $e^*$ has one end inside $f(t)$ and one end outside $f(t)$.
So $e^*$ intersects exactly one edge of $f(t)$,
and so $t \in \delta \f(e^*)$.
Second, if $t \in \delta \f(e^*)$ then $t$ contains exactly one edge from $\f(e^*)$,
and so one end of $e^*$ is inside $f(t)$ and one end outside $f(t)$.
\end{proof}

\begin{lemma}
\label{lem:interTriangles}
Let $t^* \in T^*$ and let $e^*_1,e^*_2,e^*_3$
be the three edges of $t^*$.
Then,
$$\f(e^*_1)+\f(e^*_2)+\f(e^*_3) = \delta \f(t^*).$$
\end{lemma}

\begin{proof}
There are two containments to prove.
First, if $e \in \f(e^*_1)+\f(e^*_2)+\f(e^*_3)$,
then $e$ belongs to exactly one of the sets $\f(e^*_1),\f(e^*_2),\f(e^*_3)$.
This means that one end of $f(e)$ is inside $t^*$ and one end is outside $t^*$,
so $e \in \delta \f(t^*)$.
Second, if $\f(t^*) = 0$, then we are done.
Otherwise, $\f(t^*) = \{v\}$,
and $\delta \f(t^*)$ is the set of all edges containing $v$.
For every edge $e \in \delta \f(t^*)$, the path $f(e)$ starts inside $t^*$
and ends outside $t^*$, and so $e$ intersects exactly one edge of $t^*$.
\end{proof}

\subsection{Construction of a folding map}

\begin{proof}[Proof of Theorem~\ref{thm:folding}]
We start by defining the map $\H$.
The definition is based on properties
of the coboundary and intersection maps
described in previous sub-sections.
We then show that $\H$ is indeed a folding map.

The definition of $\H$ is done in two stages,
first on $\V$ and then on $\E$.
First, define $\H : \V \to \F_2^{E}$ as follows.
For every $v^* \in \V \setminus (B/2)$, set $\H(v^*) =0$.
For every $v^* \in \V \cap (B/2)$, 
define $\H(v^*)$ as follows.
Let $u^* \in \V \setminus (B/2)$. Thus, $\f(u^*) = 0$.
There is a simple path $\{v^*, v^*_1\},\{v^*_1,v^*_2\},\ldots,
\{v^*_k, u^*\}$
from $v^*$ to $u^*$ using edges in $\E$.
Lemma~\ref{lem:interEdges} implies that
\begin{align*}
\f(v^*) & = \f(v^*) + \f(v^*_1) + \f(v^*_1)
+ \f(v^*_2) + \f(v^*_2) + \ldots
+ \f(v^*_k) + \f(v^*_k) +\f(u^*) 
\\ & = \delta \f(\{v^*,v^*_1\}) + \delta \f(\{v^*_1,v^*_2\})
+ \ldots + \delta \f(\{v^*_{k-1},v^*_k\}) + \delta \f(\{v^*_k,u^*\}) .
\end{align*}
Lemma~\ref{lem:MW+Gw} implies that there is $\H(v^*) \subset E$ so that
\begin{align*}
\f(v^*) = \delta \H(v^*) \ \  \text{and} \ \ 
\w(\H(v^*)) < \frac{3 \w(\f(v^*))}{2} \leq \frac{3c}{2}.
\end{align*}
Second, define $\H : \E \to \F_2^V$ as follows.
Let $e^* = \{v^*_1,v^*_2\} \in \E$.
If $e^* \cap (B/2) = \emptyset$, then set $\H(e^*) = 0$.
Otherwise, consider
\begin{align}
\label{eqn:a}
a = \f(e^*) + \H(v^*_1)+\H(v^*_2) \in \F_2^E.
\end{align}
Lemma~\ref{lem:interEdges} implies that
\begin{align*}
\delta a = \delta \f(e^*) + \f(v^*_1)+\f(v^*_2) = 0 .
\end{align*}
Claim~\ref{clm:sysEdges} and Lemma~\ref{lem:MW+Ggraphsw}
imply that there is $\H(e^*) \subset V$ so that
$$a = \delta \H(e^*) \ \ \text{and} \ \
\w(\H(e^*)) \leq \w(a).$$

So far the construction of the map $\H$.
It remains to prove that it is indeed a folding.
The second property of folding is clearly satisfied.
Before proving that the first property is satisfied,
we briefly discuss an upper bound on $\w(\H(e^*)) \leq \w(a)$
for $e^*,a$ from~\eqref{eqn:a}.
If $e^*$ belongs to one of the triangles in $\{t^*_v : v \in V\}$,
then (since $f$ is well-behaved) the set
$\f(e^*)$ consists only of edges that contain a single vertex $v \in V$,
so 
$$\w(\f(e^*)) \leq \w(v) + \frac{1}{n-1} \leq c + \frac{1}{n-1}.$$
Otherwise (again since $f$ is well-behaved)
every $v \in V$ belongs to at most one edge
in $\f(e^*)$ so 
$$\w(\f(e^*)) \leq \sum_{v \in V} \frac{\w(v)}{n-1} \leq \frac{1}{n-1}.$$
Overall,
\begin{align*}
\w(\H(e^*)) \leq \w(a) <  c + \frac{1}{n-1} + 2 \frac{3c}{2} = 4c  + \frac{1}{n-1}.
\end{align*}

Finally, to prove that $\H$ satisfies
the first property of folding, let $t^* \in T^*$.
Denote by $e^*_1,e^*_2,e^*_3$ the three edges of $t^*$,
and by $v^*_1,v^*_2,v^*_3$ the three vertices of $t^*$.
Consider
$$b = \f(t^*) + \H(e^*_1)+\H(e^*_2)+\H(e^*_3) \in  \F_2^V.$$
Lemma~\ref{lem:interTriangles} implies that
\begin{align*}
\delta b 
& = \delta \f(t^*) + 
\f(e^*_1) + \H(v^*_1)+\H(v^*_2) +  \\
& \qquad + \f(e^*_2) + \H(v^*_1)+\H(v^*_3)
+ \f(e^*_3) + \H(v^*_2)+\H(v^*_3) = 0.
\end{align*}
On the other hand
$$\w(b) < c + 3 \left(4c+ \frac{1}{n-1}\right) =1.$$
Claim~\ref{clm:sysVertices} hence implies that $b = 0$.
\end{proof}

\section*{Acknowledgements}

I thank Shay Moran for helpful suggestions and comments.

\bibliographystyle{plain}
\bibliography{references}

\begin{thebibliography}{1}

\bibitem{zbMATH03958121}
E.~{Boros} and Z.~{F\"uredi}.
\newblock {The number of triangles covering the center of an n-set.}
\newblock {\em {Geom. Dedicata}}, 17:69--77, 1984.

\bibitem{Dotterrer:2015aa}
D.~Dotterrer, T.~Kaufman, and U.~Wagner.
\newblock On expansion and topological overlap.
\newblock 06 2015.

\bibitem{zbMATH05800304}
M.~{Gromov}.
\newblock {Singularities, expanders and topology of maps. II: From
  combinatorics to topology via algebraic isoperimetry.}
\newblock {\em {Geom. Funct. Anal.}}, 20(2):416--526, 2010.

\bibitem{Kaufman:2014aa}
T.~Kaufman, D.~Kazhdan, and A.~Lubotzky.
\newblock Ramanujan complexes and bounded degree topological expanders.
\newblock 08 2014.

\bibitem{zbMATH06346215}
J.~{Matou\v{s}ek} and U.~{Wagner}.
\newblock {On Gromov's method of selecting heavily covered points.}
\newblock {\em {Discrete Comput. Geom.}}, 52(1):1--33, 2014.

\bibitem{zbMATH05565381}
R.~{Meshulam} and N.~{Wallach}.
\newblock {Homological connectivity of random $k$-dimensional complexes.}
\newblock {\em {Random Struct. Algorithms}}, 34(3):408--417, 2009.

\bibitem{Neumann1928}
J.~von Neumann.
\newblock Zur theorie der gesellschaftsspiele.
\newblock {\em Mathematische Annalen}, 100:295--320, 1928.

\bibitem{zbMATH03287444}
E.C. {Zeeman}.
\newblock {Seminar on combinatorial topology.}
\newblock {Paris: Institut des Hautes Etudes Scientifiques 1963. Chap. 1-6;
  Chap. 7, (1965); Chap. 8 (1966).}, 1966.

\end{thebibliography}

\end{document}